\documentclass[11pt]{amsart}

\RequirePackage[l2tabu,orthodox]{nag}
\usepackage[T1]{fontenc}
\usepackage[utf8]{inputenc}
\usepackage{lmodern}
\usepackage[english]{babel}
\usepackage{color} 
\usepackage[usenames,dvipsnames]{xcolor}
\usepackage{amsrefs}
\usepackage{amsthm,amssymb,amsmath}
\usepackage{enumerate}
\usepackage[inline]{enumitem}
\usepackage{url}
\usepackage{multicol}
\usepackage{mathrsfs}  

\newtheorem{thm}{Theorem}[section]
\newtheorem*{introthm}{Main Theorem}

\newtheorem{prop}[thm]{Proposition}
\newtheorem{cor}[thm]{Corollary}

\theoremstyle{definition}

\theoremstyle{remark}
\newtheorem{rem}[thm]{Remark}

\usepackage[margin=3cm]{geometry}

\usepackage{indentfirst}
\usepackage{microtype}

\usepackage[colorlinks=true,linkcolor=blue,citecolor=magenta]{hyperref}

\usepackage{blindtext}

\usepackage{nameref}

\DeclareMathOperator{\homeo}{\mathsf{Homeo}}

\DeclareMathOperator{\PL}{\mathsf{PL}}

\newcommand{\fix}{\operatorname{\mathsf{Fix}}}

\newcommand{\R}{\mathbb{R}}
\newcommand{\Q}{\mathbb Q}
\newcommand{\N}{\mathbb N}
\newcommand{\Z}{\mathbb Z}
\newcommand{\T}{\mathbb S^1}

\renewcommand{\setminus}{\smallsetminus}

\renewcommand{\ker}{\operatorname{\mathsf{ker}}}

\newcommand{\rot}{\mathsf{rot}}
\DeclareMathOperator{\supp}{\mathsf{Supp}}

\title{$\Z^2$-free subgroups of Thompson's $T$ are virtually free}

\author{Nicol\'as Matte Bon \and Michele Triestino}
\date{\today}

\begin{document}
\begin{abstract}
It is an open problem to determine whether surface groups can embed in Thompson's group $V$.
We prove that any finitely generated subgroup of Thompson's group $T$ without abelian groups of arbitrary high rank is either virtually abelian or virtually free. In particular, surface groups don't embed in $T$, answering a question of Belk and Moore.

	\smallskip
	
	{\noindent\footnotesize \textbf{MSC\textup{2020}:} Primary 37C85, 57M60. Secondary 37E05, 37E10.}

\end{abstract}

\maketitle

\section{Introduction}

The celebrated \emph{Thompson's group $V$} is the group of all (orientation-preserving, right-continuous) piecewise-linear (PL) bijections of the circle $\T=\R/\Z$, which are locally of the form $x\mapsto 2^kx+p/2^q$ and with breaks in $\Z[\tfrac12]/\Z$. This group has been introduced by R.~J.~Thompson in 1965, and it provided, together with the subgroup $T\le V$ of all elements defining circle homeomorphisms, the first examples of infinite simple groups which are finitely presented.
The survey by Cannon, Floyd, and Parry \cite{CFP} is a standard introduction to these and related groups. The subgroup structure of Thompson's groups $T$ and $V$ is still under investigation, and only sparse results are available; see for instance the survey by Burillo, Cleary, and R\"over \cite{MR3822284}. As stated by Bleak, Matucci, and Neunh\"{o}ffer \cite[Question 7]{BMN}, it is an open problem to determine whether a (higher genus, closed) surface group embeds in Thompson's group $V$. As Thompson's group $T$ is naturally a subgroup of $V$, Jim Belk and Justin Tatch Moore asked as an intermediary question whether surface groups embed in Thompson's group $T$. We give a negative answer to this question, by characterizing subgroups of $T$ which are \emph{$\Z\wr\Z$-free}, namely without subgroups isomorphic to the wreath product $\Z\wr\Z$.

\begin{introthm}
	Let $G$ be a finitely generated subgroup of Thompson's group $T$. Then we have the following trichotomy:
	\begin{enumerate}
	\item either $G$ is virtually abelian, or
	\item $G$ is virtually free, or
	\item $G$ contains a subgroup isomorphic to $\Z\wr\Z$.
	\end{enumerate}
\end{introthm}

\begin{rem}
	It has been proved by Alvarez \textit{et al.} in \cite{PingPong1} that any finitely generated virtually free group $G$ of orientation-preserving circle homeomorphisms (abstractly) embeds  into Thompson's group $T$, and these are exactly the finitely generated \emph{free-by-finite-cyclic groups}: there exists a free normal subgroup $H$ and a finite cyclic group $\Z_m$ such that $G/H\cong \Z_m$. On the other hand, any finitely generated virtually free group embeds in $V$ (see for instance the work of Bennett and Bleak \cite{BennettBleak}*{Theorem 1.1}).
\end{rem}

\begin{rem}
	Trying to guess a similar statement for Thompson's group $V$, we remark that one can embed in $V$ any wreath product $H\wr \Z$, with $H$ finite. There are interesting examples of finitely generated, $\Z^2$-free subgroups of $V$ which are not virtually free, such as the Houghton's groups $H_2$ and $H_3$ (the latter one is even finitely presented), but they indeed contain wreath products $H\wr\Z$ as above. See for instance the work of Antol\'in, Burillo, and Martino \cite{MR3302573} for an introduction to these groups.
\end{rem}

\begin{rem}
Note that the group $\PL(\T)$ of piecewise-linear circle homeomorphisms contains surface groups: examples of actions of surface groups with maximal Euler class (conjugate to Fuchsian groups) have been constructed by Ghys \cite{GhysAIF}; see also the related work of Minakawa \cite{Minakawa}. In this direction, the recent work of Dinamarca, Escayola, Kim, and Koberda \cite{DEKK} investigates the structure of acylindrically hyperbolic subgroups of $\PL(\T)$. 
\end{rem}

\section{Subgroups with free orbits}

Before discussing the proof of the main theorem, we start with a partial result, which holds also for particular subgroups of Thompson's $V$. For this, it will be more convenient to consider Thurston's interpretation of $V$ as the group of piecewise-$\mathsf{PSL}(2,\Z)$ (orientation-preserving, right-continuous) bijections of the projective line $\R P^1=\R\cup \{\infty\}$ with breaks in $\Q\cup \{\infty\}$ (see \cite{CFP} for more details). This piecewise-$\mathsf{PSL}(2,\Z)$ action of $V$ is topologically conjugate to the standard PL action on $\T$. Recall also that the group  $\mathsf{PSL}(2,\Z)$ of integral M\"obius transformations is isomorphic to $\Z_2*\Z_3$, and thus virtually free.

\begin{prop}\label{p.free_orbit}
	Let $G\le V$ be a finitely generated subgroup whose action on $\R P^1\cong \T$ admits free orbits. Then $G$ is virtually free.
\end{prop}

\begin{proof}
	  Fix a finite generating set $S$ for $G$, and let $S'\subset \mathsf{PSL}(2,\Z)$ be the collection of elements which locally coincide with some element from $S$ (considering Thurston's interpretation of $V$). Up to adding finitely many elements to $S'$, we can assume that the subgroup $\Gamma:=\langle S'\rangle$ is the whole group $\mathsf{PSL}(2,\Z)$.
	   With this choice, we have that for any $x\in \R P^1$, the orbit $G\cdot x$ is contained in the orbit $\Gamma\cdot x$. Moreover, we can define a map from the Schreier graph of the orbit $G\cdot x$, with respect to the generating system $S$, to the Schreier graph of the orbit $\Gamma\cdot x$, with respect to the generating set $S'$. Indeed, if $y\in G\cdot x$ and $s\in S$, there exists a unique $s'\in S'$ such that $s$ coincides with $s'$ on a right neighborhood of $y$. In particular, $s(y)=s'(y)$, so that the edge $(y,s(y))$ in the Schreier graph of $G\cdot x$ is sent to the edge $(y,s'(y))$ in the Schreier graph of $\Gamma\cdot x$. The map defined in this way is a 1-Lipschitz embedding of the Schreier graph of $G\cdot x$ into $\Gamma\cdot x$.
	  Now, it is well-known that for the action of $\Gamma=\mathsf{PSL}(2,\Z)$ on $\R P^1$, the Schreier graph of the orbit of any point is a \emph{quasi-tree} (that is, quasi-isometric to a tree). Therefore, the map defined above gives a 1-Lipschitz embedding of the Schreier graph of $G\cdot x$ into a quasi-tree. As we have learnt from Romain Tessera, this implies that the Schreier graph of $G\cdot x$ itself is a quasi-tree. One way to see this, is by using the notion of \emph{separation profile} introduced by Benjamini, Schramm, and Tim\'ar in \cite{BSTseparation}: for a graph $X$, this is a function $\mathsf{sep}_X\colon \N\to \N$ that for any $n\in \N$ gives the supremum over
	all subgraphs of size $n$, of the number of vertices needed to be removed from the
	subgraph, in order to cut it to connected subgraphs of size at most $n/2$. If a graph $Y$ coarsely embeds in a graph $X$ (in particular, if there exists a $1$-Lipschitz embedding of $Y$ in $X$), then $\mathsf{sep}_Y(n)=O(\mathsf{sep}_X(n))$ as $n\to \infty$, see \cite[Lemma 1.3]{BSTseparation}. When $X$ is a vertex-transitive, bounded-degree, connected graph (conditions all satisfied by Schreier graphs of actions of finitely generated groups), then the separation profile of $X$ is bounded if and only if $X$ is a quasi-tree, as observed by Hume and Mackay \cite[Theorem 1.2]{HMseparation}. This yields the desired conclusion.
	  
	  Now, if $x\in \R P^1$ has trivial stabilizer for the action of $G$, the Schreier graph of $G\cdot x$ can be identified with the Cayley graph of the group $G$. Since a finitely generated group is virtually free if and only if its Cayley graph is a quasi-tree, we deduce that $G$ is virtually free.\qedhere

\end{proof}

\begin{rem}\label{rem:HSZ}
One can extract from the proof above the fact that Schreier graphs defined by the action on $\T$ by a finitely generated subgroup of $V$ are quasi-trees. This result is Theorem~A in the recent work of Hyde, Skipper, and Zaremsky \cite{HSZ}; the proof is essentially the same, with the main difference being that the remark of Tessera provides a big shortcut.
\end{rem}

\begin{rem}
The reader may compare the statement of Proposition \ref{p.free_orbit} with the work of Bennett and Bleak \cite{BennettBleak}: Lemma 2.11 there gives the same conclusion assuming instead the existence of a non-empty open interval in $\T$ whose $G$-orbit is free. Their proof relies on the characterization by Muller and Schupp of finitely generated virtually free groups as those with context-free word problem.
\end{rem}

\section{$\Z\wr\Z$-free groups of PL circle homeomorphisms}

The next partial results are about subgroups of $\PL(\T)$ which do not contain copies of $\Z\wr\Z$. Here we will highly rely on structural properties of groups acting on the circle, and specific dynamical properties of groups of PL circle homeomorphisms. The analogue approach for studying subgroups of $V$ seems sensibly more challenging.

The first fundamental ingredient is a classical result by Guba and Sapir \cite{GubaSapir} on groups of piecewise linear homeomorphisms of the interval.

\begin{thm}[Guba--Sapir]
	Let $G\le \PL([0,1])$ be a non-abelian subgroup of the group of orientation-preserving PL homeomorphisms of the interval $[0,1]$. Then $G$ contains a subgroup isomorphic to $\Z\wr \Z$.
\end{thm}


We immediately deduce the following.

\begin{cor}\label{c.BS}
	Let $G\le \PL(\T)$ be a $\Z\wr\Z$-free subgroup of the group of orientation-preserving PL homeomorphisms of the circle. Then, the stabilizer of any point in $G$ is abelian.
\end{cor}

We next recall some basic facts and terminology for group actions on the circle (see for instance Ghys \cite{Ghys} for an introduction).
For any group action on the circle, either the action admits a finite orbit, or there is a unique non-empty closed invariant subset of the circle which is minimal with respect to inclusion, called the \emph{minimal set} of the action. In the latter case, the minimal set is either the whole circle, or (a subset homeomorphic to) a Cantor set.
We say that a subgroup $G\le \homeo_+(\T)$ of the group of orientation-preserving circle homeomorphisms is \emph{elementary} if it preserves a Borel probability measure on the circle $\T$. Note that when $G$ is amenable, in particular a cyclic group, the action is always elementary. The \emph{rotation number} of a circle homeomorphism $g$ can be defined by the expression
\begin{equation}\label{eq.rotation}
	\rot(g)=\mu([x,g(x)))\pmod{\Z},\tag{$\star$}
\end{equation}
where $\mu$ is any $g$-invariant Borel probability measure, and $x\in \T$ a point. It is a crucial fact that the  value $\rot(g)$ does not depend on the choices of the invariant measure $\mu$ and the base point $x$. The rotation number, as a function $\rot\colon\homeo_+(\T)\to \R/\Z$ is not a homomorphism, but whenever a subgroup $G\le \homeo_+(\T)$ is elementary, the expression \eqref{eq.rotation}, for a common invariant measure $\mu$, ensures that the restriction of $\rot$ to $G$ is a homomorphism. A homeomorphism $g$ has rational rotation number $\rot(g)=p/q$ if and only if it admits a periodic orbit of order $q$. In particular, a circle homeomorphism $g$ has rotation number $\rot(g)=0$ if and only if it admits a fixed point: more precisely, any point in the support of an invariant Borel probability measure will be fixed.
Hence, when $G$ is elementary, the kernel of $\rot\colon G\to \R/\Z$ has global fixed points on the circle.

We will also use the fact that the classical Denjoy's theorem holds for piecewise-linear homeomorphisms, as proved by Herman in \cite{Herman}*{\S VI.4}:
\begin{thm}[Herman]\label{t.Denjoy}
	If $g\colon \T\to \T$ is a PL homeomorphism with irrational rotation number, then there exists a unique $g$-invariant Borel probability measure on $\T$, and it has full support.
\end{thm}

Then, we have the following description of elementary subgroups $G\le \PL(\T)$.

\begin{prop}\label{p.elementary}
	Let $G\le \PL(\T)$ be an elementary, finitely generated $\Z\wr\Z$-free subgroup. Then $G$ is virtually abelian.
\end{prop}

\begin{proof}
	As we are assuming that $G$ is elementary, the rotation number $\rot\colon G\to \R/\Z$ defines a homomorphism. Let us first assume that there exists $g\in G$ with irrational rotation number; after Theorem \ref{t.Denjoy}, there exists a unique $g$-invariant Borel probability measure $\mu$ on $\T$, which is therefore the unique $G$-invariant Borel probability measure. After the previous discussion, every element in the kernel of $\rot$ fixes any point in the support of $\mu$. But after Theorem \ref{t.Denjoy}, the support of $\mu$ is the whole circle, so $\rot$ is injective. We deduce that $G$ is isomorphic to a subgroup of $\R/\Z$, hence abelian.
Assume next that $\rot(G)\subset \Q/\Z$. Since $G$ is finitely generated, the image $\rot(G)\subset \Q/\Z$ is finite. The subgroup $G_0=\ker (\rot)$ has global fixed points, so by Corollary \ref{c.BS}, $G_0$ is abelian.
\end{proof}

We next study the case when $G$ is non-elementary.

\begin{prop}\label{l.non-elementary}
	Let $G\le \PL(\T)$ be a non-elementary, countable, $\Z\wr\Z$-free subgroup, and let $\Lambda\subset \T$ be the minimal set for $G$. Then there exists a countable subset $\Delta\subset \Lambda$ such that the $G$-orbit of any point $x\in \Lambda\setminus \Delta$ is free.
\end{prop}

\begin{proof}
	As $G$ is countable, it is enough to prove that any non-trivial element of $G$ has at most finitely many fixed points in $\Lambda$. Assume this is not the case, then there exists a non-trivial element $g\in G$ such that the subset $\fix(g)=\{x\in \T:g(x)=x\}$ contains points of $\Lambda$ in its interior. We will write $\supp(g)=\T\setminus \fix(g)$ for the complement.
	We will use a classical result by Antonov \cite{Antonov}, rediscovered later by Margulis (see Ghys \cite[\S 5.2]{Ghys}): there exists a finite order homeomorphism $\phi\in \homeo_+(\T)$, commuting with $G$, such that the induced action of $G$ on the quotient circle $C:=\T/\langle \phi\rangle$ is \emph{proximal} in restriction to $L:=\Lambda/\langle \phi\rangle$: for any proper open subsets $U,V\subset C$ intersecting $L$, there exists an element $h\in G$ such that $h(C\setminus V)\subset U$.
	Let us take two disjoint open intervals $U,V\subset \fix(g)/\langle \phi\rangle$ intersecting the minimal set $L$. Let $h\in G$ be an element provided by the theorem of Antonov for this choice of $U$ and $V$, and remark that $h$ has fixed points in $U$ and $V$ but not outside these intervals. Going back to the circle $\T$, up to consider a positive power of $h$ (bounded by the order of $\phi$), we can assume that $h$ has fixed points on $\T$, which are all fixed by $g$. Moreover, the choice of $h$ ensures that $h(\supp(g))\subset \fix(g)$, and this implies that $g$ and $h$ generate a subgroup isomorphic to $\Z\wr\Z$ (otherwise, one simply checks that $h$ and $g$ cannot commute, and apply Corollary \ref{c.BS}).
\end{proof}

\section{Proof of the main theorem}

	Let $G\le T$ be a finitely generated $\Z\wr\Z$-free subgroup.
	If $G$ is elementary, Proposition \ref{p.elementary} gives that $G$ is virtually abelian. When $G$ is non-elementary, by Proposition \ref{l.non-elementary} we get that the $G$-action on $\T$ has (uncountably many) free orbits, so that Proposition \ref{p.free_orbit} allows us to conclude that $G$ is virtually free.

{\small \subsection*{Acknowledgments}
	This note was mainly written after a visit of M.T. at the Department of Mathematics at Cornell University in February 2022, and M.T. thanks the warm hospitality received there. The result was presented for the first time in August 2026 at Purdue University in the conference \emph{PL homeomorphisms and related topics}, and we thank the participants for their feedback, especially Thomas Koberda and Matt Zaremsky.
		
	M.T. is partially supported by the project ANR
	AnoDyn (ANR-24-CE40-5065-01), and both authors are partially supported by the projects ``Small spaces under actions'' (ECOS Sud -- Chili) and ``Interactions between groups and dynamics'' (MATH AmSud).}

\bibliographystyle{plain}

\bibliography{biblio.bib}

\medskip

\noindent\textit{Nicol\'as Matte Bon\\
	CNRS, Université Claude Bernard Lyon 1,\\
	Centrale Lyon, INSA Lyon, Université Jean Monnet,\\
	ICJ, UMR CNRS 5208, 69622 Villeurbanne, France\\}
\href{mailto:mattebon@math.univ-lyon1.fr}{mattebon@math.univ-lyon1.fr}

\smallskip

\noindent\textit{Michele Triestino\\
	Aix-Marseille Université, CNRS,\
	I2M \& Institut Universitaire de France,\\
	3 place Victor Hugo, 13331 Marseille Cedex 3, France\\}
\href{mailto:michele.triestino@univ-amu.fr}{michele.triestino@univ-amu.fr}

\end{document}